\documentclass[12pt,a4paper]{article}

\usepackage[utf8]{inputenc}
\usepackage[T1]{fontenc}
\usepackage[english]{babel}
\usepackage{amsmath, amssymb, amsthm}
\usepackage{algorithm}
\usepackage{algorithmic}
\usepackage[round,authoryear]{natbib}
\usepackage{placeins}
\usepackage{graphicx}
\usepackage{hyperref}
\usepackage{bm}
\usepackage{geometry}
\usepackage{algorithm}
\usepackage{algorithmic}
\usepackage{listings}
\newtheorem{proposition}{Proposition}[section]
\newtheorem{corollary}{Corollary}[section]

\title{Curvature--Residual Geometry in Bregman Regression}
\author{Ky Vu \\ Mathematics Department, FPT University \\ \texttt{kyvk2@fe.edu.vn}}

\begin{document}

\maketitle

\begin{abstract}
We consider linear regression models fitted by minimizing Bregman losses of the form
\[
\frac{1}{n}\sum_{i=1}^n
\left[
\phi(y_i)-\phi(x_i^\top\theta)
-\phi'(x_i^\top\theta)
\bigl(y_i-x_i^\top\theta\bigr)
\right],
\]
where \(y_i\) is the observed response and \(x_i^\top\theta\) is the linear prediction. Even when the generating potential \(\phi\) is strongly convex, the resulting regression objective may be nonconvex in \(\theta\). This creates a gap between the convexity of the generating potential and the optimization geometry of the fitted model.

The Hessian can be written as a weighted Gram matrix whose weights depend on the derivatives of the potential and the current residuals. This representation gives simple conditions for local strong convexity, smoothness, and conditional linear convergence of gradient descent.

For the quadratic--quartic potential, we derive an exact scalar convexity condition, identify the interval of negative curvature, and obtain local and global sufficient conditions for positive curvature. Numerical experiments indicate that these scalar conditions may fail while the full Hessian remains positive definite at the evaluated points. They also indicate that the range of tested gradient-descent step sizes leading to convergence decreases as the quartic parameter grows.

These results characterize how residual-dependent curvature interacts with the generating potential and the design matrix in reverse Bregman regression.
\end{abstract}
\section{Introduction}
\label{sec:introduction}

\subsection{How Convex Bregman Losses Become Nonconvex}
\label{subsec:curvature-mechanism}

Let \(\phi\in C^3(I)\) be a strictly convex function on an open interval \(I\subseteq\mathbb{R}\). The Bregman divergence generated by \(\phi\) is
\[
D_\phi(y\|m)
=
\phi(y)-\phi(m)-\phi'(m)(y-m).
\]
Given observations \((x_i,y_i)_{i=1}^n\), with \(x_i\in\mathbb{R}^d\) and \(y_i\in I\), we study the regression objective
\[
L(\theta)
=
\frac{1}{n}\sum_{i=1}^n
D_\phi(y_i\|x_i^\top\theta)
\]
on the parameter domain
\[
\Theta
=
\left\{
\theta\in\mathbb{R}^d:
x_i^\top\theta\in I
\text{ for every }i
\right\}.
\]

The quadratic choice \(\phi(u)=u^2/2\) recovers ordinary least squares. For a general potential, however, the objective may be nonconvex in \(\theta\), even when \(\phi\) is uniformly strongly convex and the prediction map is linear. The source of this nonconvexity is a residual-dependent term in the curvature of the loss.

To see this, let \(m_i=x_i^\top\theta\) and \(r_i=y_i-m_i\). Direct differentiation gives
\[
\nabla L(\theta)
=
-\frac{1}{n}
\sum_{i=1}^n
\phi''(m_i)r_i x_i,
\]
and
\[
\nabla^2L(\theta)
=
\frac{1}{n}
\sum_{i=1}^n
\left[
\phi''(m_i)-\phi'''(m_i)r_i
\right]
x_ix_i^\top.
\]

Thus the Hessian is a weighted sum of the rank-one matrices \(x_ix_i^\top\), with weights
\[
w_i(\theta)
=
\phi''(m_i)-\phi'''(m_i)r_i.
\]
Under uniform strong convexity, the first term is positive, but the residual-dependent term has no fixed sign. It may reduce the local curvature or make an individual weight negative. Therefore, strong convexity of \(\phi\) alone does not imply convexity of the regression objective.

If the weights remain uniformly positive and bounded on a convex region \(\mathcal D\subseteq\Theta\), namely,
\[
0<\mu\le w_i(\theta)\le M
\]
for every \(i\) and every \(\theta\in\mathcal D\), then the curvature of \(L\) is controlled by that of the least-squares problem:
\[
\mu\frac{X^\top X}{n}
\preceq
\nabla^2L(\theta)
\preceq
M\frac{X^\top X}{n},
\qquad
\theta\in\mathcal D.
\]
Hence, when \(X\) has full column rank, \(L\) is strongly convex and smooth on \(\mathcal D\). Gradient descent then converges linearly, provided that the iterates remain in this region. If \(X\) is rank deficient, the same statements hold after restricting the parameter space to \(\ker(X)^\perp\).

To make the general curvature mechanism explicit, we next consider a simple nonquadratic perturbation of least squares. The quartic family
\[
\phi_{\lambda,\alpha}(u)
=
\frac{\lambda}{2}u^2+\frac{\alpha}{4}u^4,
\qquad
\lambda>0,\quad \alpha\ge0,
\]
is particularly useful for this purpose. When \(\alpha=0\), it reduces to the quadratic potential and hence to ordinary least squares. For \(\alpha>0\), the potential remains uniformly strongly convex, but its curvature is no longer constant. It therefore provides a tractable example in which the interaction between potential curvature and residuals can create nonconvexity while still allowing an exact analysis.

The numerical value of \(\alpha\) depends on the scale of the response. Under the rescaling \(y,m\mapsto cy,cm\), an equivalent parametrization with \(\lambda\) fixed requires \(\alpha\mapsto\alpha/c^2\). In the real-data experiments below, the responses are therefore standardized using training-set statistics before fitting.

Indeed,
\[
\phi_{\lambda,\alpha}''(u)
=
\lambda+3\alpha u^2
\ge\lambda,
\]
yet the corresponding regression objective need not be convex. Its Hessian weights are
\[
w_i(\theta)
=
\lambda+3\alpha m_i^2-6\alpha m_ir_i
=
\lambda+3\alpha(m_i-r_i)^2-3\alpha r_i^2.
\]
Hence
\[
w_i(\theta)\ge\lambda-3\alpha r_i^2.
\]
For \(\alpha>0\), the condition
\[
\|r(\theta)\|_\infty
<
\sqrt{\frac{\lambda}{3\alpha}}
\]
is sufficient for all the weights to remain positive.

The residual-based condition above is local and only sufficient. A different rearrangement gives an exact scalar characterization. Using \(r_i=y_i-m_i\), we obtain
\[
w_i(\theta)
=
\lambda-\alpha y_i^2
+
9\alpha\left(m_i-\frac{y_i}{3}\right)^2.
\]
For a fixed response \(y_i\), the map
\[
m\mapsto D_{\phi_{\lambda,\alpha}}(y_i\|m)
\]
is convex on \(\mathbb{R}\) if and only if \(\alpha y_i^2\le\lambda\). When this condition fails, the scalar loss has an explicit interval of negative curvature. Applying the condition to every observation gives the sufficient global bound
\[
\alpha\|y\|_\infty^2<\lambda,
\]
under which the regression objective is globally strongly convex when \(X\) has full column rank, and globally strongly convex on \(\ker(X)^\perp\) otherwise. Failure of this condition does not imply that the full Hessian is indefinite, since positive contributions from other observations may compensate for negative scalar weights.

As \(\alpha\) increases, the objective may become poorly conditioned or develop negative curvature. We therefore examine continuation from \(\alpha=0\), using the stationary point computed at one stage to initialize the next. The implicit function theorem gives only a local justification, so continuation is treated here as a numerical warm-start strategy rather than as a globally convergent method.

\subsection{Related work}
\label{subsec:related-work}

Bregman divergences were introduced in connection with projection methods for convex programming and are now standard tools in convex analysis and optimization \citep{Bregman1967,BauschkeBorwein1997}. Bregman projections satisfy generalized Pythagorean relations, while Legendre duality leads to the dually flat structures studied in information geometry
\citep{BauschkeBorwein1997,AmariNagaoka2000,Amari2010}. They also form the geometric basis of mirror descent and related first-order methods \citep{NemirovskiYudin1983,BeckTeboulle2003}.

Bregman divergences also appear as prediction losses and distortion measures in statistical learning. Expected Bregman loss of the form \(D_\phi(Y\|m)\) is minimized by the conditional mean under standard integrability assumptions, and Bregman divergences are closely connected with clustering, exponential-family models, proper scoring rules, and statistical risk
\citep{BanerjeeGuoWang2005,BanerjeeMeruguDhillonGhosh2005,GneitingRaftery2007,ReidWilliamson2011}.

The order of the arguments is important for convexity. A Bregman divergence is convex in its first argument, but need not be convex in its second argument. Conditions for second-argument convexity have been studied in connection with proper composite losses and divergence inequalities \citep{ReidWilliamson2010,PainskyWornell2019}. These works emphasize that second-argument convexity imposes additional restrictions on the generating function. The present paper instead studies directly how failure of this convexity appears in a linear regression objective through residual-dependent Hessian weights.

Questions of curvature and loss of convexity also arise in several related model classes. In single-neuron models, the interaction between the loss and the transfer function may produce difficult nonconvex landscapes
\citep{HelmboldKivinenWarmuth1995,KivinenHelmboldWarmuth1999, AuerHerbsterWarmuth1995}. In generalized linear models, noncanonical links introduce additional curvature through the nonlinear mean parameterization \citep{NelderWedderburn1972,McCullaghNelder1989}. Relative smoothness and Bregman first-order methods provide a framework for objectives whose gradients are not globally Lipschitz \citep{BauschkeBolteTeboulle2017,LuFreundNesterov2018,
BolteSabachTeboulleVaisbourd2018}.

The setting considered here differs in the source of nonconvexity while retaining a linear prediction map. The prediction \(x_i^\top\theta\) is linear, but the curvature depends on the residual because the prediction enters the second argument of the Bregman divergence. Our focus is the resulting weighted-Gram Hessian and the local optimization geometry of the regression objective. We also use continuation in the quartic parameter as a numerical warm-start strategy, following the general idea of homotopy methods for parameterized nonconvex problems \citep{SuzumuraEtAl2017}.

\subsection{Main contributions}

The main contributions of the paper are as follows.

\begin{enumerate}
    \item We identify the residual-dependent curvature mechanism in
    reverse Bregman regression. The Hessian has a weighted-Gram form,
    with scalar weights determined jointly by the derivatives of the
    potential and the current residuals. Uniform bounds on these weights
    yield local strong-convexity and smoothness estimates, together with
    a conditional linear convergence guarantee for gradient descent
    when the iterates remain in the certified region.

    \item For the quadratic-quartic potential, we derive a
    residual-based sufficient condition for local positive curvature and
    an exact scalar curvature characterization. The latter identifies an
    explicit interval of negative curvature when
    \(\alpha y_i^2>\lambda\), and yields the sufficient condition
    \(
    \alpha\|y\|_\infty^2<\lambda
    \)
    for global strong convexity when \(X\) has full column rank, and on
    \(\ker(X)^\perp\) otherwise.

    \item Numerical experiments compare the scalar curvature conditions
    with the spectrum of the full Hessian. In the tested instances, the
    sufficient conditions may fail while the Hessian at the computed
    approximately stationary point remains positive definite. The
    experiments also indicate that the range of tested gradient-descent
    step sizes leading to convergence decreases as the quartic parameter
    increases.

    \item We also examine continuation from the least-squares problem at
    \(\alpha=0\) as a numerical warm-start strategy. A local
    implicit-function argument describes the continuation of
    nondegenerate stationary points under small changes in \(\alpha\).
    In the tested synthetic and real-data settings, continuation gives
    diagnostics similar to direct optimization but requires more
    iterations and function evaluations.
\end{enumerate}

\section{General curvature bounds and local convergence}
\label{sec:general-curvature}

We now formalize the curvature mechanism described in the introduction. Let \(I\subseteq\mathbb{R}\) be an open interval, let \(\phi\in C^3(I)\) be strictly convex, and suppose that \(y_i\in I\) for every \(i\). Define
\[
\Theta
=
\left\{
\theta\in\mathbb{R}^d:
x_i^\top\theta\in I
\text{ for every }i
\right\},
\]
and consider
\[
L(\theta)
=
\frac1n\sum_{i=1}^n
D_\phi\bigl(y_i\|x_i^\top\theta\bigr),
\qquad
\theta\in\Theta.
\]

Write
\[
m_i(\theta)=x_i^\top\theta,
\qquad
r_i(\theta)=y_i-m_i(\theta),
\]
and let \(X\in\mathbb{R}^{n\times d}\) denote the design matrix with rows \(x_i^\top\). As noted in Section~\ref{subsec:curvature-mechanism},
\[
\nabla L(\theta)
=
-\frac1n\sum_{i=1}^n
\phi''\bigl(m_i(\theta)\bigr)r_i(\theta)x_i,
\qquad
\nabla^2L(\theta)
=
\frac1nX^\top W(\theta)X,
\]
where
\[
W(\theta)
=
\operatorname{diag}
\bigl(
w_1(\theta),\ldots,w_n(\theta)
\bigr),
\qquad
w_i(\theta)
=
\phi''\bigl(m_i(\theta)\bigr)
-
\phi'''\bigl(m_i(\theta)\bigr)r_i(\theta).
\]
Thus the Hessian is a weighted Gram matrix whose scalar weights depend jointly on the potential and the current residuals. Throughout this section, \(\mathcal D\subseteq\Theta\) is a convex set.

Suppose that, for some \(0<\mu\le M<\infty\),
\[
\mu\le w_i(\theta)\le M,
\qquad
i=1,\ldots,n,\quad \theta\in\mathcal D.
\]
Then
\[
\mu\frac{X^\top X}{n}
\preceq
\nabla^2L(\theta)
\preceq
M\frac{X^\top X}{n},
\qquad
\theta\in\mathcal D.
\]
If \(X\) has full column rank, set
\[
\mu_X
=
\mu\lambda_{\min}\left(\frac{X^\top X}{n}\right),
\qquad
M_X
=
M\lambda_{\max}\left(\frac{X^\top X}{n}\right).
\]
Then \(L\) is \(\mu_X\)-strongly convex and \(M_X\)-smooth on \(\mathcal D\), with
\[
\frac{M_X}{\mu_X}
=
\frac{M}{\mu}
\kappa\left(\frac{X^\top X}{n}\right).
\]
If \(X\) is rank deficient, the same conclusions hold on \(\ker(X)^\perp\), since \(L\) is constant along \(\ker(X)\).

A sufficient condition for these weight bounds follows from bounds on the derivatives of the potential and the residuals. Suppose that \(m_i(\theta)\in J\subseteq I\) for every \(i\) and \(\theta\in\mathcal D\), and that
\[
0<c\le\phi''(u)\le C,
\qquad
|\phi'''(u)|\le\Gamma,
\qquad
u\in J.
\]
If
\[
|r_i(\theta)|\le\rho,
\qquad
i=1,\ldots,n,\quad \theta\in\mathcal D,
\]
then
\[
c-\Gamma\rho
\le
w_i(\theta)
\le
C+\Gamma\rho,
\]
and hence
\[
(c-\Gamma\rho)\frac{X^\top X}{n}
\preceq
\nabla^2L(\theta)
\preceq
(C+\Gamma\rho)\frac{X^\top X}{n}.
\]
In particular, if \(\Gamma>0\) and
\(
\rho<\frac{c}{\Gamma},
\)
the Hessian is positive definite when \(X\) has full column rank, and positive definite on \(\ker(X)^\perp\) otherwise. This condition is sufficient but not necessary, since \(X^\top W(\theta)X\) may remain positive semidefinite even when some scalar weights are negative.

The preceding bounds give the following conditional convergence result.

\begin{corollary}
\label{thm:gradient-descent-convergence}
Assume that \(X\) has full column rank and that the weight bounds above hold on the convex set \(\mathcal D\). Let
\(\theta^\star\in\mathcal D\) satisfy
\[
\nabla L(\theta^\star)=0,
\]
and let
\[
\theta_{t+1}
=
\theta_t-\eta\nabla L(\theta_t),
\qquad
0<\eta\le\frac1{M_X}.
\]
If the iterates remain in \(\mathcal D\), then
\[
L(\theta_t)-L(\theta^\star)
\le
(1-\eta\mu_X)^t
\bigl[
L(\theta_0)-L(\theta^\star)
\bigr].
\]
Moreover, \(\theta^\star\) is the unique minimizer of \(L\) on \(\mathcal D\).
\end{corollary}

\begin{proof}
The result follows from the \(\mu_X\)-strong convexity and \(M_X\)-smoothness of \(L\) on \(\mathcal D\), together with the
standard gradient-descent estimate.
\end{proof}

With \(\eta=1/M_X\), the convergence factor is
\(
1-\frac{\mu_X}{M_X}.
\)
When the weight bounds hold only on a restricted region, the result is conditional on the gradient-descent trajectory remaining in that region; no invariance of \(\mathcal D\) is asserted.

The analysis therefore reduces local curvature and conditioning to bounds on the scalar weights \(w_i(\theta)\). The next section derives more explicit conditions for the quadratic--quartic potential.

\section{Quartic Bregman Regression}
\label{sec:quartic}

To make the curvature mechanism explicit, we consider the quadratic--quartic family as a simple nonquadratic perturbation of least squares. It preserves uniform strong convexity of the generating potential, while allowing the residual-dependent term in the Hessian to create negative curvature. At the same time, its polynomial structure permits an exact scalar analysis.

We therefore specialize to
\[
\phi_{\lambda,\alpha}(u)
=
\frac{\lambda}{2}u^2+\frac{\alpha}{4}u^4,
\qquad
\lambda>0,\quad \alpha\ge0,
\]
and write \(L_\alpha\) for the corresponding regression objective. Its derivatives are
\[
\phi_{\lambda,\alpha}'(u)=\lambda u+\alpha u^3,
\qquad
\phi_{\lambda,\alpha}''(u)=\lambda+3\alpha u^2,
\qquad
\phi_{\lambda,\alpha}'''(u)=6\alpha u.
\]
Hence we have
\(
\phi_{\lambda,\alpha}''(u)\ge\lambda,
\)
which implies that the potential itself is uniformly strongly convex. The scalar divergence and empirical objective are
\[
D_{\phi_{\lambda,\alpha}}(y\|m)
=
\frac{\lambda}{2}(y-m)^2
+
\frac{\alpha}{4}
\left(y^4-4m^3y+3m^4\right)
\]
and
\[
L_\alpha(\theta)
=
\frac{\lambda}{2n}\|y-X\theta\|_2^2
+
\frac{\alpha}{4n}
\sum_{i=1}^n
\left(
y_i^4-4m_i(\theta)^3y_i+3m_i(\theta)^4
\right).
\]
The gradient and Hessian are
\[
\nabla L_\alpha(\theta)
=
-\frac1n\sum_{i=1}^n
\left(
\lambda+3\alpha m_i(\theta)^2
\right)
r_i(\theta)x_i
\]
and
\[
\nabla^2L_\alpha(\theta)
=
\frac1n\sum_{i=1}^n
w_i(\theta;\alpha)x_ix_i^\top,
\]
where
\[
w_i(\theta;\alpha)
=
\lambda+3\alpha m_i(\theta)^2
-6\alpha m_i(\theta)r_i(\theta).
\]
Although the potential is uniformly strongly convex, these weights may become negative because of the residual-dependent term.

Since
\[
L_\alpha(\theta)
\ge
\frac{\lambda}{2n}\|y-X\theta\|_2^2,
\]
the objective is coercive and attains a global minimum when \(X\) has full column rank. In the rank-deficient case, the same holds on \(\ker(X)^\perp\).

We now derive two complementary curvature bounds for the quartic objective. One is expressed in terms of the current residuals and gives a local sufficient condition; the other is expressed in terms of the
responses and yields an exact scalar characterization. For the first, completing the square gives
\[
w_i(\theta;\alpha)
=
\lambda
+
3\alpha\bigl(m_i(\theta)-r_i(\theta)\bigr)^2
-
3\alpha r_i(\theta)^2.
\]
Therefore,
\[
w_i(\theta;\alpha)
\ge
\lambda-3\alpha r_i(\theta)^2.
\]

For \(\alpha>0\), the pointwise condition
\(
\|r(\theta)\|_\infty
<
\sqrt{\frac{\lambda}{3\alpha}}
\)
is sufficient for all scalar weights to be positive, with
\[
\nabla^2L_\alpha(\theta)
\succeq
\left(
\lambda-3\alpha\|r(\theta)\|_\infty^2
\right)
\frac{X^\top X}{n}.
\]
Thus the Hessian is positive definite when \(X\) has full column rank, and its restriction to \(\ker(X)^\perp\) is positive definite otherwise. More generally, if
\[
\sup_{\theta\in\mathcal D}
\|r(\theta)\|_\infty
\le\rho
\qquad\text{and}\qquad
3\alpha\rho^2<\lambda,
\]
then
\[
\nabla^2L_\alpha(\theta)
\succeq
\left(
\lambda-3\alpha\rho^2
\right)
\frac{X^\top X}{n},
\qquad
\theta\in\mathcal D.
\]

A different rearrangement yields an exact scalar characterization. Using \(r_i=y_i-m_i\),
\[
w_i(\theta;\alpha)
=
\lambda-\alpha y_i^2
+
9\alpha
\left(
m_i(\theta)-\frac{y_i}{3}
\right)^2.
\]
For a fixed \(y\), let
\[
\ell_y(m)
=
D_{\phi_{\lambda,\alpha}}(y\|m).
\]

\begin{proposition}
\label{prop:quartic-scalar-convexity}
For \(\alpha>0\),
\[
\min_{m\in\mathbb{R}}\ell_y''(m)
=
\lambda-\alpha y^2.
\]
Consequently,
\[
\ell_y
\text{ is convex on }\mathbb{R}
\quad\Longleftrightarrow\quad
\alpha y^2\le\lambda.
\]
If \(\alpha y^2>\lambda\), then \(\ell_y''(m)<0\) exactly on
\[
m^-(y)<m<m^+(y),
\qquad
m^\pm(y)
=
\frac{y\pm\sqrt{y^2-\lambda/\alpha}}{3}.
\]
\end{proposition}

\begin{proof}
The identity
\[
\ell_y''(m)
=
\lambda-\alpha y^2
+
9\alpha
\left(
m-\frac{y}{3}
\right)^2
\]
shows that the minimum is attained at \(m=y/3\). When \(\alpha y^2>\lambda\), solving \(\ell_y''(m)=0\) gives the stated endpoints, and the quadratic is negative between them.
\end{proof}

Applying the scalar lower bound to every observation gives
\[
w_i(\theta;\alpha)
\ge
\lambda-\alpha\|y\|_\infty^2.
\]
Hence
\[
\alpha\|y\|_\infty^2<\lambda
\]
is sufficient for
\[
\nabla^2L_\alpha(\theta)
\succeq
\left(
\lambda-\alpha\|y\|_\infty^2
\right)
\frac{X^\top X}{n},
\qquad
\theta\in\mathbb{R}^d.
\]
Thus \(L_\alpha\) is globally strongly convex when \(X\) has full column rank, and globally strongly convex on \(\ker(X)^\perp\) otherwise.

Both the residual-based and response-based conditions are sufficient at the regression level. Failure of the response-based condition does not imply that the full Hessian is indefinite, since positive contributions from other observations may compensate for negative scalar weights. At the boundary
\[
\alpha\|y\|_\infty^2=\lambda,
\]
the preceding argument yields only a positive-semidefinite lower bound. The full Hessian may nevertheless remain positive definite, since the scalar lower bounds need not be attained simultaneously and the remaining positive-weight design vectors may still span the parameter space.

We finally record an upper curvature bound on regions with bounded predictions. For \(\alpha>0\) and \(X\neq0\), the Hessian is unbounded
on \(\mathbb{R}^d\), so \(\nabla L_\alpha\) is not globally Lipschitz.
If
\[
|m_i(\theta)|\le B,
\qquad
i=1,\ldots,n,\quad \theta\in\mathcal D,
\]
then
\[
w_i(\theta;\alpha)
=
\lambda+9\alpha m_i(\theta)^2
-6\alpha y_i m_i(\theta)
\le
\lambda+9\alpha B^2+6\alpha B\|y\|_\infty.
\]
Consequently,
\[
\nabla^2L_\alpha(\theta)
\preceq
\left(
\lambda+9\alpha B^2+6\alpha B\|y\|_\infty
\right)
\frac{X^\top X}{n},
\qquad
\theta\in\mathcal D.
\]

Combined with either of the preceding lower bounds, this gives explicit strong convexity and smoothness constants on \(\mathcal D\). The conditional gradient-descent result of Section~\ref{sec:general-curvature} then applies as long as the iterates remain in this region.

\section{Continuation in the Quartic Parameter}
\label{sec:continuation}

As \(\alpha\) increases, the quartic objective moves away from least squares and may become more poorly conditioned or nonconvex. This suggests using the solution of an easier problem at a smaller value of
\(\alpha\) to initialize the next one. We therefore examine continuation in \(\alpha\) as a numerical warm-start strategy.

At \(\alpha=0\),
\[
L_0(\theta)
=
\frac{\lambda}{2n}\|y-X\theta\|_2^2,
\]
so the procedure starts from a least-squares solution. If \(X\) has full column rank, this solution is
\[
\theta^\star(0)
=
(X^\top X)^{-1}X^\top y;
\]
otherwise, we restrict the parameter space to \(\ker(X)^\perp\) and use the minimum-norm solution \(X^\dagger y\).

The analysis below is local. It describes when a stationary point varies smoothly under a small change in \(\alpha\), but it does not guarantee that the path extends to a prescribed target value.

To describe the local behavior of stationary points, observe that the gradient depends affinely on \(\alpha\):
\[
\nabla L_\alpha(\theta)
=
\nabla L_0(\theta)+\alpha G(\theta),
\qquad
G(\theta)
=
-\frac3n\sum_{i=1}^n
m_i(\theta)^2r_i(\theta)x_i.
\]
Hence
\[
\partial_\alpha\nabla L_\alpha(\theta)=G(\theta).
\]

\begin{proposition}
\label{prop:local-solution-path}
Let \(E=\ker(X)^\perp\), and let \(\alpha_0\ge0\) and \(\theta_0\in E\). Suppose that
\[
\nabla_E L_{\alpha_0}(\theta_0)=0
\]
and that the restricted Hessian
\[
\nabla_E^2L_{\alpha_0}(\theta_0):E\to E
\]
is invertible. Then there exists a neighborhood \(U\) of \(\alpha_0\) and a unique \(C^1\) map
\[
\theta^\star:U\cap[0,\infty)\to E
\]
such that
\[
\theta^\star(\alpha_0)=\theta_0,
\qquad
\nabla_E L_\alpha\bigl(\theta^\star(\alpha)\bigr)=0,
\qquad
\alpha\in U\cap[0,\infty).
\]
Along this local stationary path,
\[
\frac{d\theta^\star(\alpha)}{d\alpha}
=
-
\left[
\nabla_E^2L_\alpha
\bigl(\theta^\star(\alpha)\bigr)
\right]^{-1}
\partial_\alpha\nabla_E L_\alpha
\bigl(\theta^\star(\alpha)\bigr).
\]
\end{proposition}

\begin{proof}
Define
\[
F(\theta,\alpha)=\nabla_E L_\alpha(\theta),
\qquad
(\theta,\alpha)\in E\times\mathbb R.
\]
For the quadratic--quartic family, \(F\) is continuously differentiable in both variables. Its polynomial dependence on \(\alpha\) also defines \(F\) on an open neighborhood of \(\alpha_0\), including when \(\alpha_0=0\).

By assumption,
\[
F(\theta_0,\alpha_0)=0,
\qquad
\partial_\theta F(\theta_0,\alpha_0)
=
\nabla_E^2L_{\alpha_0}(\theta_0)
\]
is invertible. The implicit function theorem therefore yields a unique local \(C^1\) map \(\theta^\star(\alpha)\) satisfying
\[
\theta^\star(\alpha_0)=\theta_0,
\qquad
F\bigl(\theta^\star(\alpha),\alpha\bigr)=0.
\]

Differentiating this identity gives
\[
\nabla_E^2L_\alpha\bigl(\theta^\star(\alpha)\bigr)
\frac{d\theta^\star(\alpha)}{d\alpha}
+
\partial_\alpha\nabla_E L_\alpha
\bigl(\theta^\star(\alpha)\bigr)
=
0,
\]
which yields the stated formula.
\end{proof}

For the quadratic--quartic family,
\[
\partial_\alpha\nabla L_\alpha(\theta)
=
-\frac3n\sum_{i=1}^n
m_i(\theta)^2r_i(\theta)x_i.
\]
Since each \(x_i\in E\), this is also \(\partial_\alpha\nabla_E L_\alpha(\theta)\). Therefore,
\[
\frac{d\theta^\star(\alpha)}{d\alpha}
=
\frac3n
\left[
\nabla_E^2L_\alpha
\bigl(\theta^\star(\alpha)\bigr)
\right]^{-1}
\sum_{i=1}^n
m_i\bigl(\theta^\star(\alpha)\bigr)^2
r_i\bigl(\theta^\star(\alpha)\bigr)x_i.
\]

If the restricted Hessian is positive definite at \((\theta_0,\alpha_0)\), then continuity implies that nearby stationary points remain strict local minimizers on \(E\) for sufficiently small changes in \(\alpha\). Moreover,
\[
\left\|
\frac{d\theta^\star(\alpha)}{d\alpha}
\right\|_2
\le
\frac{
\left\|
\partial_\alpha\nabla_E L_\alpha
\bigl(\theta^\star(\alpha)\bigr)
\right\|_2
}{
\lambda_{\min}
\left(
\nabla_E^2L_\alpha
\bigl(\theta^\star(\alpha)\bigr)
\right)
},
\]
whenever the restricted Hessian is positive definite. Thus the local stationary path may become sensitive to \(\alpha\) when its smallest restricted Hessian eigenvalue is positive but close to zero.

This local result motivates the following warm-start procedure. Let
\[
0=\alpha_0<\alpha_1<\cdots<\alpha_K
=\alpha_{\mathrm{target}}.
\]
At each stage, the point computed for \(L_{\alpha_{k-1}}\) is used to initialize a local solver for \(L_{\alpha_k}\).

\begin{algorithm}[t]
\caption{Continuation warm start in \(\alpha\)}
\label{alg:alpha-continuation}
\begin{algorithmic}[1]

\REQUIRE Data \(X,y\), parameter \(\lambda>0\), schedule
\(0=\alpha_0<\cdots<\alpha_K\), tolerance \(\varepsilon>0\)

\STATE Compute a least-squares solution
\(\theta_0\in\arg\min_\theta L_0(\theta)\)

\FOR{\(k=1,\ldots,K\)}

    \STATE Initialize the local solver at \(\theta_{k-1}\)

    \STATE Apply the solver to \(L_{\alpha_k}\), and let \(\theta_k\)
    denote its final iterate

    \STATE Record
    \[
    \|\nabla_E L_{\alpha_k}(\theta_k)\|_2,
    \qquad
    \lambda_{\min}
    \bigl(\nabla_E^2L_{\alpha_k}(\theta_k)\bigr),
    \qquad
    \|r(\theta_k)\|_\infty
    \]

    \STATE Declare the stage successful if
    \[
    \|\nabla_E L_{\alpha_k}(\theta_k)\|_2\le\varepsilon
    \]

\ENDFOR

\RETURN \(\theta_K\)

\end{algorithmic}
\end{algorithm}

Proposition~\ref{prop:local-solution-path} suggests that the previously computed stationary point may provide a useful initialization when \(\alpha_k-\alpha_{k-1}\) is small and the restricted Hessian remains nonsingular. The Hessian spectrum provides a local diagnostic. At a stationary point, a positive smallest restricted eigenvalue implies a strict local minimum on \(E\), whereas a negative smallest eigenvalue rules out local minimality. A small eigenvalue in absolute value indicates proximity to singularity. The residual condition
\[
\|r(\theta_k)\|_\infty
<
\sqrt{\frac{\lambda}{3\alpha_k}}
\]
guarantees positivity of all scalar Hessian weights, although it is only a sufficient condition for positive definiteness of the full restricted Hessian.

If the local solve deteriorates or the Hessian approaches singularity, intermediate values of \(\alpha\) may be inserted into the schedule. This refinement is heuristic. An invertible but indefinite Hessian may define a stationary path without preserving local minimality, while at a singular Hessian the implicit function theorem no longer guarantees a locally unique path. The procedure therefore does not guarantee that a stationary path extends to \(\alpha_{\mathrm{target}}\), that local minimality is preserved, or that the final point is a global minimizer.

Under
\[
\alpha\|y\|_\infty^2<\lambda,
\]
the objective has a unique stationary point when \(X\) has full column rank, and a unique stationary point on \(E\) otherwise. Continuation is not required for existence or uniqueness in this certified regime. Beyond it, continuation is used only as a numerical warm-start strategy.

Section~\ref{sec:experiments} compares continuation with direct optimization using stationarity, the Hessian spectrum, and the residual condition as diagnostics.

\section{Numerical Experiments}
\label{sec:experiments}

This section examines the curvature conditions derived in Section~\ref{sec:quartic} and the continuation strategy introduced in Section~\ref{sec:continuation}, with emphasis on optimization behavior rather than predictive performance.

Unless stated otherwise, we set \(\lambda=1\). At each evaluation point, we record the Hessian
\[
H_\alpha(\theta)=\nabla^2L_\alpha(\theta),
\]
its smallest eigenvalue \(\lambda_{\min}(H_\alpha(\theta))\), the residual norm \(\|r(\theta)\|_\infty\), the gradient norm \(\|\nabla L_\alpha(\theta)\|_2\), and the minimum scalar weight
\[
w_{\min}(\theta)
=
\min_{1\le i\le n} w_i(\theta;\alpha).
\]
The minimum scalar weight is used to compare positivity of the individual weights with positive definiteness of the full Hessian. In experiments repeated over independent trials, we also report the empirical frequency
\[
\widehat{\mathbb P}
\left(
\lambda_{\min}(H_\alpha(\theta))<0
\right),
\]
since a positive average smallest eigenvalue does not imply positive definiteness in every trial.

The synthetic experiments use standardized Gaussian designs and independent random seeds. The number of trials is stated in each experiment. Quantities evaluated at approximately stationary points are averaged only over runs satisfying the prescribed gradient tolerance, whereas success rates are computed over all runs.

\subsection{Scalar loss geometry}
\label{subsec:scalar-experiment}

We first illustrate the scalar loss
\[
\ell_y(m)=D_{\phi_{\lambda,\alpha}}(y\|m).
\]
By Proposition~\ref{prop:quartic-scalar-convexity},
\[
\ell_y''(m)
=
\lambda-\alpha y^2
+
9\alpha\left(m-\frac{y}{3}\right)^2,
\]
so the transition between global scalar convexity and nonconvexity occurs at
\(
\alpha y^2=\lambda.
\)

We set \(\lambda=1\) and \(y=2\), for which
\(
\alpha_\star=\frac{\lambda}{y^2}=\frac14.
\)
Figure~\ref{fig:scalar-loss-geometry} plots \(\ell_y(m)\) and \(\ell_y''(m)\) for
\(
\alpha\in\{0.1,0.25,0.5\},
\)
corresponding to the subcritical, critical, and supercritical cases. In the supercritical case, the endpoints
\[
m^\pm(y)
=
\frac{y\pm\sqrt{y^2-\lambda/\alpha}}{3}
\]
are also shown.

\begin{figure}[htbp]
\centering
\includegraphics[width=0.95\textwidth]{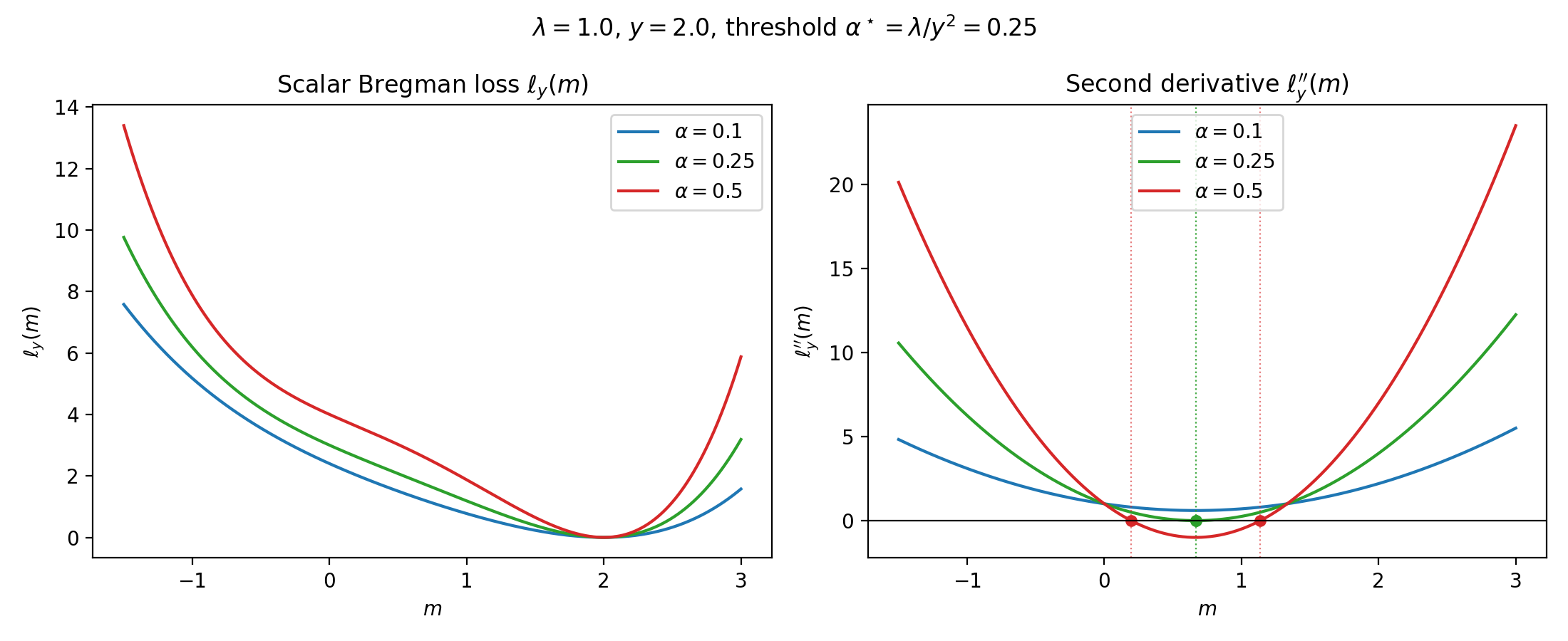}
\caption{Scalar quartic Bregman loss and its second derivative below,
at, and above the threshold \(\alpha y^2=\lambda\). In the
supercritical case, \(\ell_y''(m)\) is negative exactly on
\((m^-(y),m^+(y))\).}
\label{fig:scalar-loss-geometry}
\end{figure}

Below the threshold, \(\ell_y\) is globally strongly convex. At the threshold, it remains convex, with \(\ell_y''(y/3)=0\). Above the threshold, it has negative curvature exactly on \((m^-(y),m^+(y))\).

\subsection{Synthetic curvature map}
\label{subsec:synthetic-curvature}

We next compare the scalar sufficient conditions with the observed curvature of the full regression Hessian. We set \(n=200\) and \(d=30\), generate \(X\in\mathbb{R}^{n\times d}\) with independent standard Gaussian entries, and standardize its columns. The responses are generated as
\[
y=X\theta^\star+\sigma\varepsilon,
\qquad
\varepsilon\sim\mathcal N(0,I_n),
\]
where \(\theta^\star\) is a random unit vector. We vary \(\alpha\in[0,0.6]\) and \(\sigma\in[0.1,2.0]\) over an \(11\times11\) grid, using \(200\) independently generated instances at each grid point.

For each instance, curvature is evaluated both at the least-squares solution \(\theta_{\mathrm{LS}}\) and at a point \(\hat\theta_\alpha\) obtained by running L-BFGS-B from \(\theta_{\mathrm{LS}}\). A run is classified as approximately stationary when
\[
\|\nabla L_\alpha(\hat\theta_\alpha)\|_2
\]
is below the prescribed tolerance.

Figure~\ref{fig:phase-diagram} reports the mean smallest Hessian eigenvalue and the empirical probability
\[
\widehat{\mathbb P}
\bigl(
\lambda_{\min}(H_\alpha(\theta))<0
\bigr)
\]
over the tested grid. The zero contours of the mean global and residual lower bounds are included for comparison.

\begin{figure}[htbp]
\centering
\includegraphics[width=\textwidth]{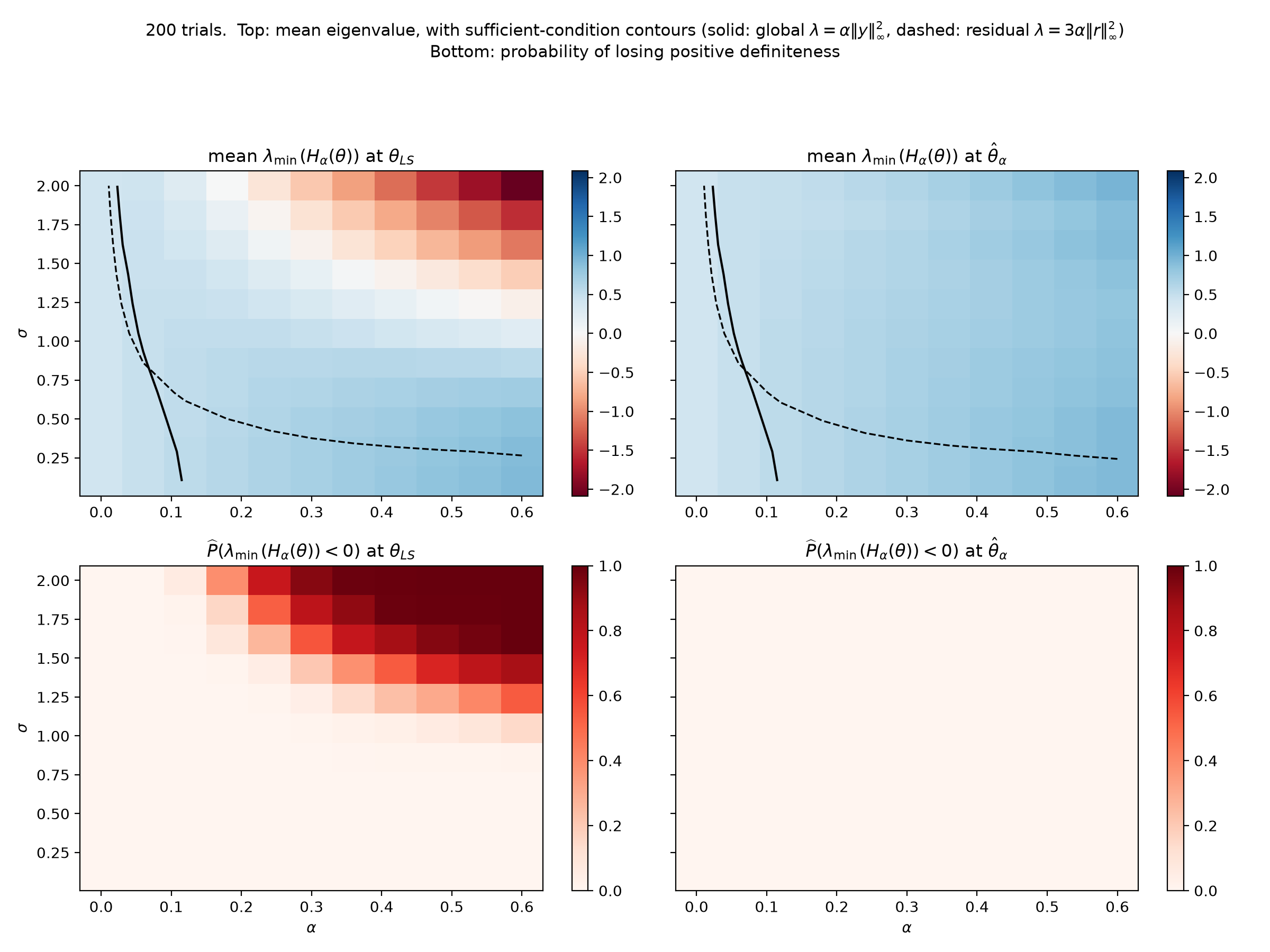}
\caption{Top: mean smallest Hessian eigenvalue over the tested
\((\alpha,\sigma)\)-grid at \(\theta_{\mathrm{LS}}\) and
\(\hat\theta_\alpha\), with the zero contours of the mean global and
residual sufficient bounds. Bottom: empirical probability of a
negative smallest Hessian eigenvalue
(\(n=200\), \(d=30\), \(200\) trials per grid point).}
\label{fig:phase-diagram}
\end{figure}

At \(\theta_{\mathrm{LS}}\), negative curvature becomes increasingly common as both \(\alpha\) and \(\sigma\) increase. The empirical probability first reaches \(0.5\) at \(\sigma=1.24\) and \(\alpha=0.6\), and the crossing occurs at smaller values of \(\alpha\) for larger noise levels. In contrast, no negative smallest eigenvalue was observed at \(\hat\theta_\alpha\) in any of the \(200\) trials at any tested grid point. For a fixed grid point, the rule of three gives an approximate pointwise upper confidence bound of \(3/200=1.5\%\); this is not a simultaneous bound over the full grid.

Table~\ref{tab:phase-variability} reports selected distributional summaries at \(\alpha=0.6\). At \(\theta_{\mathrm{LS}}\), the distribution changes from predominantly positive to predominantly negative as \(\sigma\) increases. At \(\hat\theta_\alpha\), the
empirical negative-eigenvalue probability is zero and the \(10\)th percentile remains positive at every reported noise level.

\begin{table}[!htbp]
\centering
\small
\setlength{\tabcolsep}{4pt}
\begin{tabular}{rlrrrrr}
\hline
\(\sigma\) & Evaluation point & Mean & SE & Median & \(q_{10}\)
& \(\widehat{\mathbb P}(\mathrm{neg})\) \\
\hline
0.10 & \(\theta_{\mathrm{LS}}\)      & 0.924  & 0.006 & 0.928  & 0.824  & 0.000 \\
0.10 & \(\hat\theta_\alpha\)         & 0.925  & 0.006 & 0.932  & 0.826  & 0.000 \\
0.86 & \(\theta_{\mathrm{LS}}\)      & 0.562  & 0.014 & 0.584  & 0.344  & 0.015 \\
0.86 & \(\hat\theta_\alpha\)         & 0.863  & 0.011 & 0.867  & 0.669  & 0.000 \\
1.24 & \(\theta_{\mathrm{LS}}\)      & \(-0.115\) & 0.029 & \(-0.026\) & \(-0.693\) & 0.540 \\
1.24 & \(\hat\theta_\alpha\)         & 0.824  & 0.017 & 0.815  & 0.508  & 0.000 \\
1.62 & \(\theta_{\mathrm{LS}}\)      & \(-1.093\) & 0.050 & \(-0.931\) & \(-2.078\) & 1.000 \\
1.62 & \(\hat\theta_\alpha\)         & 0.910  & 0.023 & 0.910  & 0.521  & 0.000 \\
2.00 & \(\theta_{\mathrm{LS}}\)      & \(-2.085\) & 0.068 & \(-1.917\) & \(-3.370\) & 1.000 \\
2.00 & \(\hat\theta_\alpha\)         & 0.975  & 0.028 & 0.960  & 0.527  & 0.000 \\
\hline
\end{tabular}
\caption{Selected summaries of
\(\lambda_{\min}(H_\alpha(\theta))\) over \(200\) trials at
\(\alpha=0.6\). Here \(q_{10}\) denotes the \(10\)th percentile and
\(\widehat{\mathbb P}(\mathrm{neg})\) the empirical probability of a
negative smallest eigenvalue.}
\label{tab:phase-variability}
\end{table}
\FloatBarrier

These results show that failure of the scalar sufficient conditions can occur well before negative curvature is typically observed in the full Hessian at the evaluated points. They also indicate a substantial difference between curvature at the least-squares initialization and at the approximately stationary point returned by local optimization.

\subsection{Gradient-descent stability}
\label{subsec:gd-stability}

We next examine how the quartic parameter \(\alpha\) interacts with the gradient-descent step size \(\eta\). Using \((n,d)=(200,30)\) and fixing \(\sigma=0.5\), we run
\[
\theta_{t+1}
=
\theta_t-\eta\nabla L_\alpha(\theta_t)
\]
from \(\theta_0=\theta_{\mathrm{LS}}\) over an \(11\times11\) grid with \(\alpha\in[0,0.6]\) and logarithmically spaced \(\eta\in[0.01,1]\). Each configuration is repeated over \(200\) independent trials for at most \(2000\) iterations.

A run is classified as converged if
\[
\|\nabla L_\alpha(\theta_t)\|_2\le10^{-4}
\]
within the iteration limit, divergent if a nonfinite value appears or \(\|\theta_t\|_2>10^6\), and stalled otherwise. We report the empirical fraction of converged trials at each grid point.

Figure~\ref{fig:gd-stability} shows the resulting convergence probabilities together with examples of converged, stalled, and divergent trajectories. For each displayed trajectory, it also reports the objective value, the smallest Hessian eigenvalue, and the maximum residual norm.

\begin{figure}[htbp]
\centering
\includegraphics[width=\textwidth]{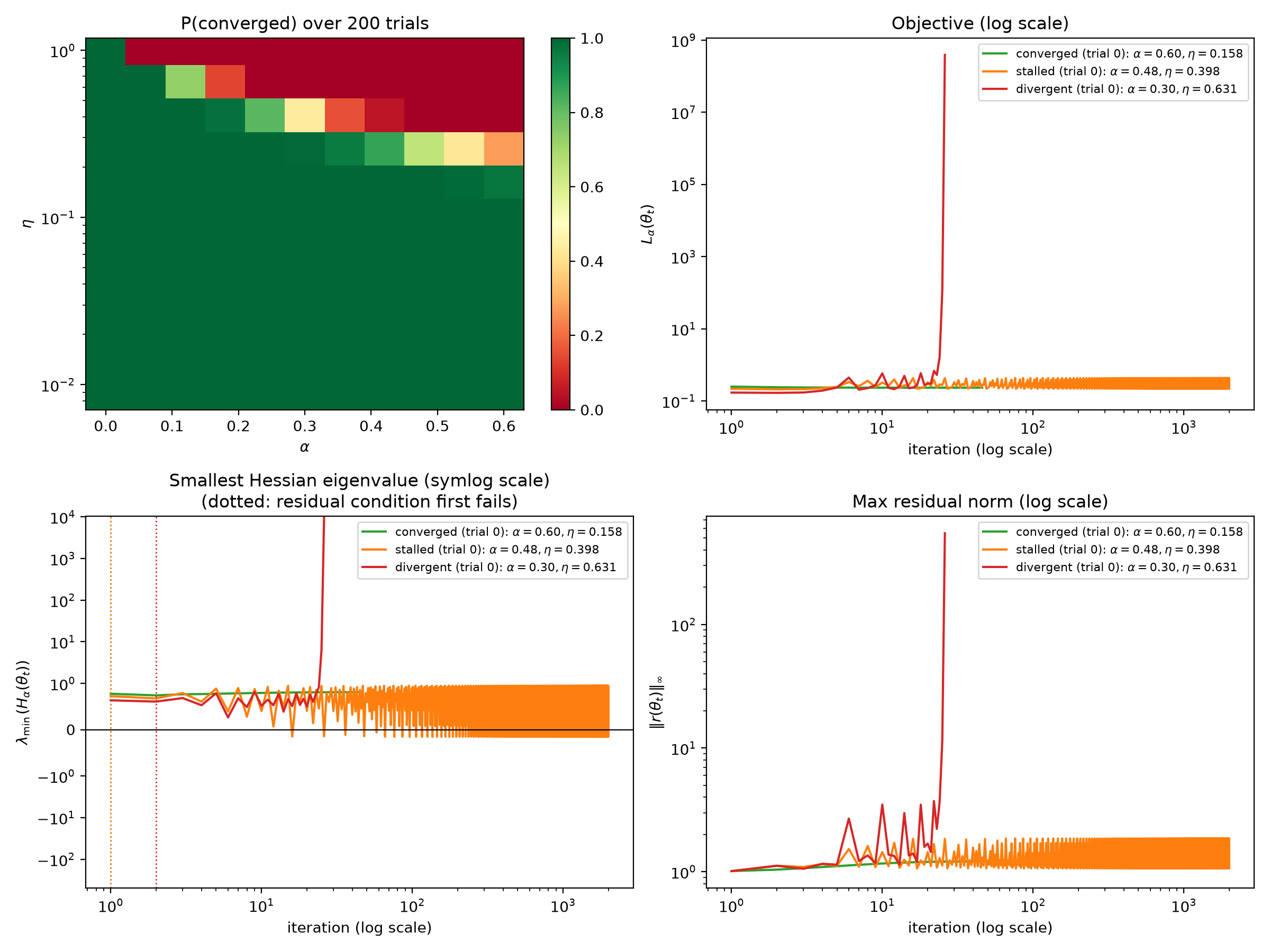}
\caption{Empirical gradient-descent convergence probability over the
tested \((\alpha,\eta)\)-grid, together with examples of converged,
stalled, and divergent trajectories
(\(200\) trials per grid point).}
\label{fig:gd-stability}
\end{figure}

Table~\ref{tab:gd-threshold} reports, for each \(\alpha\), the largest
tested step size for which the empirical convergence probability is at
least \(0.5\). This value decreases from \(1\) at \(\alpha=0\) to
\(0.159\) at \(\alpha=0.6\).

\begin{table}[htbp]
\centering
\begin{tabular}{cc}
\hline
\(\alpha\)
&
Largest tested \(\eta\) with
\(\widehat{\mathbb P}(\mathrm{converge})\ge0.5\)
\\
\hline
0.00 & 1.000 \\
0.12 & 0.631 \\
0.24 & 0.398 \\
0.36 & 0.251 \\
0.48 & 0.251 \\
0.60 & 0.159 \\
\hline
\end{tabular}
\caption{Largest tested step size for which the empirical convergence
probability is at least \(0.5\), based on \(200\) trials per grid point.
Selected values of \(\alpha\) are shown. This finite-grid quantity is
not an exact stability boundary.}
\label{tab:gd-threshold}
\end{table}

The decrease in the tested threshold is qualitatively consistent with the upper curvature bound in Section~\ref{sec:quartic}, which increases with \(\alpha\) on a fixed bounded prediction region. The experiment does not identify an exact admissible step-size boundary, since the curvature varies along the trajectory and the hypotheses of
Corollary~\ref{thm:gradient-descent-convergence} need not hold globally.

The displayed trajectories also show that instability may occur under different local curvature patterns. Some trajectories oscillate or diverge before entering a region of negative curvature, whereas others reach points where
\[
\lambda_{\min}\bigl(H_\alpha(\theta_t)\bigr)\le0.
\]
The residual condition typically fails earlier and is therefore used only as a sufficient diagnostic, not as an exact stability boundary.

\subsection{Direct optimization versus continuation}
\label{subsec:continuation-experiment}

We compare direct optimization at a target value \(\alpha_{\mathrm{target}}\) with the continuation warm-start procedure of Section~\ref{sec:continuation}. The synthetic experiment uses \(n=150\), \(d=20\), \(\sigma=0.5\), and \(200\) independent trials. For each trial, we set
\[
\alpha_{\mathrm{target}}
=
c\,\alpha_{\mathrm{glob}},
\qquad
c\in\{0.5,1,2\},
\]
where
\[
\alpha_{\mathrm{glob}}
=
\frac{\lambda}{\|y\|_\infty^2}
\]
is computed separately for that trial. These choices lie below, at, and above the scalar sufficient threshold, but need not define distinct curvature regimes for the full Hessian.

Continuation uses \(K=5\) equally spaced values between \(0\) and \(\alpha_{\mathrm{target}}\), with L-BFGS-B initialized from the point computed at the preceding stage. Direct optimization uses the same solver initialized at \(\theta_{\mathrm{LS}}\).

A run is classified as approximately stationary when its final gradient norm is at most the prescribed tolerance. Objective values, gradient norms, iteration counts, function evaluations (\texttt{nfev}), and smallest Hessian eigenvalues are reported as mean \(\pm\) one standard deviation over approximately stationary runs.

\begin{table}[htbp]
\centering
\footnotesize
\setlength{\tabcolsep}{3pt}
\begin{tabular}{llcrrrrr}
\hline
Regime & Method & Stat. & Objective
& Gradient\(^\dagger\) & Iterations & \texttt{nfev}
& \(\lambda_{\min}(H)\) \\
\hline
below    & direct       & 1.000 & \(0.129\pm0.017\) & \(1.5\pm0.9\)
& \(7.4\pm0.7\)  & \(9.5\pm0.7\)  & \(0.518\pm0.044\) \\
below    & continuation & 0.990 & \(0.129\pm0.017\) & \(2.1\pm1.8\)
& \(28.5\pm2.6\) & \(38.7\pm2.7\) & \(0.518\pm0.044\) \\
boundary & direct       & 1.000 & \(0.146\pm0.020\) & \(1.6\pm1.1\)
& \(8.3\pm0.7\)  & \(10.4\pm0.9\) & \(0.574\pm0.054\) \\
boundary & continuation & 0.970 & \(0.146\pm0.020\) & \(2.1\pm1.7\)
& \(32.2\pm2.7\) & \(42.4\pm2.7\) & \(0.574\pm0.054\) \\
above    & direct       & 0.990 & \(0.179\pm0.028\) & \(2.2\pm1.5\)
& \(9.4\pm0.8\)  & \(11.6\pm0.9\) & \(0.673\pm0.074\) \\
above    & continuation & 0.980 & \(0.178\pm0.027\) & \(2.2\pm1.2\)
& \(36.3\pm2.6\) & \(47.0\pm2.9\) & \(0.672\pm0.075\) \\
\hline
\end{tabular}
\caption{Direct optimization and continuation on synthetic data over
\(200\) independent trials. The regimes are defined relative to the
scalar sufficient threshold
\(\alpha_{\mathrm{glob}}=\lambda/\|y\|_\infty^2\).
Here \emph{Stat.} denotes the approximately stationary rate. Objective
values, gradient norms, iteration counts, function evaluations, and
Hessian eigenvalues are averaged over approximately stationary runs.
Every approximately stationary output has a positive smallest Hessian
eigenvalue. \(^\dagger\)Gradient norms are reported in units of
\(10^{-5}\).}
\label{tab:direct-vs-continuation-synthetic}
\end{table}

Conditional on approximate stationarity, the two methods give nearly identical objective values and smallest Hessian eigenvalues in all three regimes. This aggregate agreement does not imply that they reach the same stationary point in every trial. Their stationarity rates differ by only \(0.01\) to \(0.03\). Since no paired test of the binary stationarity outcomes is performed, these differences are not interpreted as evidence of a systematic reliability difference.

The clear difference is computational cost. Continuation requires approximately three to four times as many iterations and function evaluations because it solves several intermediate problems. Thus the experiment provides no evidence of a computational advantage for continuation in the tested instances. Possible benefits under stronger nonconvexity or different initializations are not evaluated here.

\paragraph{Real-data comparison.}

We repeat the comparison on the California Housing and Diabetes datasets from \texttt{scikit-learn}. For each dataset, we use \(5\) random \(80/20\) train--test splits. Features and responses are standardized using training-set statistics, and an intercept is added after feature standardization. For each training split, we set
\[
\alpha_{\mathrm{target}}
=
\alpha_{\mathrm{glob}}
\]
and report test RMSE in the original response units. Values are reported as mean \(\pm\) one standard deviation over the \(5\) splits. The purpose of this experiment is to compare the two optimization procedures, not to assess a predictive advantage of the quartic loss.

\begin{table}[htbp]
\centering
\footnotesize
\setlength{\tabcolsep}{3pt}
\begin{tabular}{llrrrrr}
\hline
Dataset & Method & Gradient & Iterations & \texttt{nfev}
& \(\lambda_{\min}(H)\) & RMSE \\
\hline
California & direct
& \((1.3\pm0.7)\times10^{-5}\) & \(19.8\pm2.3\) & \(24.4\pm3.4\)
& \(0.0816\pm0.0012\) & \(0.735\pm0.004\) \\
California & continuation
& \((2.7\pm2.7)\times10^{-5}\) & \(80.2\pm4.7\) & \(103.2\pm8.0\)
& \(0.0816\pm0.0012\) & \(0.735\pm0.004\) \\
Diabetes & direct
& \((2.4\pm2.4)\times10^{-5}\) & \(17.2\pm3.0\) & \(21.2\pm3.0\)
& \(0.0134\pm0.0004\) & \(55.86\pm1.92\) \\
Diabetes & continuation
& \((1.6\pm0.4)\times10^{-5}\) & \(84.8\pm6.1\) & \(99.8\pm4.4\)
& \(0.0134\pm0.0004\) & \(55.86\pm1.92\) \\
\hline
\end{tabular}
\caption{Direct optimization and continuation on two real regression
datasets over \(5\) random train--test splits. Values are reported as
mean \(\pm\) one standard deviation. All runs satisfy the stationarity
tolerance, and test RMSE is reported in the original response units.}
\label{tab:direct-vs-continuation-real}
\end{table}

The real-data results show the same computational pattern as the synthetic experiment. Direct optimization and continuation give the same smallest Hessian eigenvalues and test RMSE to the reported precision, while continuation requires approximately four to five times as many iterations and function evaluations.

At
\[
\alpha_{\mathrm{target}}=\alpha_{\mathrm{glob}},
\]
the response-based lower bound is only positive semidefinite. Nevertheless, every approximately stationary output has a positive smallest Hessian eigenvalue and is therefore a strict local minimum. This is only a local certificate and does not establish global optimality. With \(5\) splits per dataset, the experiment is not intended to resolve small differences in reliability.

\section{Discussion and Conclusion}
\label{sec:conclusion}

This paper studies the gap between convexity of a generating potential and convexity of the associated reverse Bregman regression objective. When the prediction enters the second argument of the divergence, the Hessian depends on both the residuals and the design matrix.

For the quadratic--quartic family, we derived residual-based and response-based sufficient conditions for positive curvature, together with an exact scalar characterization of the negative-curvature region. These conditions are not necessary at the regression level, since the full Hessian combines the scalar weights through the design vectors. The numerical experiments illustrate this distinction: the scalar conditions may fail while the Hessian at the computed approximately stationary point remains positive definite.

The gradient-descent experiments indicate that the range of tested step sizes leading to convergence decreases as the quartic parameter increases. Continuation from least squares gives a natural warm start, but its theoretical justification is only local. In the tested synthetic and real-data settings, continuation produced diagnostics similar to direct optimization while requiring more iterations and function evaluations.

A natural next step is to develop sharper, design-dependent curvature conditions based directly on
\[
X^\top W(\theta)X
\]
rather than on the individual scalar weights. Other open questions include the statistical role and data-driven selection of the quartic parameter, extensions to nonlinear prediction maps, and continuation under stronger nonconvexity.

Overall, the quadratic--quartic model provides a simple setting in which the interaction among the potential, residuals, and design matrix can be analyzed explicitly.

\section*{Acknowledgments}

The author acknowledges the use of ChatGPT and Claude AI for editorial assistance and limited support during the preparation of numerical experiments. All experimental designs, numerical results, and source code were independently checked and validated by the author. The author assumes full responsibility for the originality, correctness, and final content of the manuscript.

\bibliographystyle{plainnat}
\bibliography{references}

\end{document}